\documentclass[11pt]{article}
\usepackage[totalwidth=13.0cm,totalheight=20.0cm]{geometry}
\usepackage{latexsym,amsthm,amsmath,amssymb,url}
\usepackage[ruled, linesnumbered]{algorithm2e}
\usepackage{enumerate}
\usepackage{tikz}

\usetikzlibrary{decorations.pathreplacing}

\newtheorem{lemma}{Lemma}

\newtheorem{definition}{Definition}
\newtheorem{theorem}{Theorem}

\newtheorem{conj}{Conjecture}
\usepackage[backref]{hyperref}
\usepackage{authblk}
\title{The Paired Domination Number of Cubic Graphs}
\author[1]{Bin Sheng\thanks{Email: shengbinhello@nuaa.edu.cn}}
\affil[1]{College of Computer Science and Technology, Nanjing University of Aeronautics and Astronautics, Collaborative Innovation Center of Novel Software Technology and Industrialization, Nanjing, Jiangsu, 211106, PR China}
\author[2]{Changhong Lu\thanks{Email: chlu@math.ecnu.edu.cn}}
\affil[2]{Department of Mathematics, Shanghai Key Laboratory of PMMP, East China Normal University, Shanghai, China}
\begin{document}

\maketitle

\begin{abstract}

Let $G$ be a simple undirected graph with no isolated vertex. A \textit{paired dominating set} of $G$ is a dominating set which induces a subgraph that has a perfect matching. The paired domination number of $G$, denoted by $\gamma_{pr}(G)$, is the size of its smallest paired dominating set.

Goddard and Henning conjectured that $\gamma_{pr}(G)\leq 4n/7$ holds for every graph $G$ with $\delta(G)\geq 3$, except the Petersen Graph. In this paper, we prove this conjecture for cubic graphs.  

\end{abstract}
\section{Introduction}

Dominating set is one of the most classic problems in graph theory.  Many variants of it have been studied due to its wide applications. There are excellent books for this topic, like~\cite{haynes2017domination,DBLP:books/daglib/0093827}. In this paper, we study the paired domination problem.

Hayes and Slater  raised the notion of  paired domination in~\cite{DBLP:journals/networks/HaynesS98}, as a  model for the problem of assigning security guards that can protect each other. It has been studied from many perspectives~\cite{DBLP:journals/ipl/ChenLZ09,DBLP:journals/dam/ChengKN07,DBLP:journals/gc/EgawaFT13,DBLP:journals/gc/FavaronH04,DBLP:journals/jco/HenningP05}.

Computing the paired-domination number of a graph has been shown to be NP-complete~\cite{DBLP:journals/networks/HaynesS98}.
Chen et al.~\cite{DBLP:journals/jco/ChenLZ10} further proved that it is NP-complete for bipartite graphs, chordal graphs, and even for split graphs. Thus obtaining tight upper bounds on the paired-domination number of a graph is an interesting problem.

For general graphs, Haynes and Slater~\cite{DBLP:journals/networks/HaynesS98} bound the paired-domination number with respect to the number of vertices in the graph.
\begin{theorem}(\cite{DBLP:journals/networks/HaynesS98})\label{thm1}
If $G$ is a connected graph with $n(\geq 3)$ vertices, then $\gamma_{pr}\leq n-1$, the equality holds if and only if $G$ is $C_3,C_5$ or a subdivided star.
\end{theorem}
If we require that $\delta(G)\geq 2$, then the following theorem improves the result in Theorem~\ref{thm1}.
\begin{theorem}(\cite{DBLP:journals/networks/HaynesS98})\label{thm2}
If $G$ is a connected graph with $n(\geq 6)$ vertices and $\delta(G)\geq 2$, then $\gamma_{pr}(G)\leq 2n/3$.
\end{theorem}
The result in Theorem~\ref{thm2} can be further strengthened if we increase the size bound of the graph.

\begin{theorem}(\cite{DBLP:journals/jco/Henning07})\label{thm3}
If $G$ is a connected graph with $n(\geq 10)$ vertices and $\delta(G)\geq 2$, then $\gamma_{pr}\leq 2(n-1)/3$. The equality holds for infinitely many graphs.
\end{theorem}

Graph with at least 14 vertices and minimum degree at least 2 are also considered in~\cite{DBLP:journals/jco/Henning07}.

There is also a series of results for graphs with $\delta(G)\geq 3$.

\begin{theorem}(\cite{DBLP:journals/appml/ChenSC08})\label{thm4}
If $G$ is a cubic graph with $n$ vertices, then $\gamma_{pr}(G)\leq 3n/5$.
\end{theorem}
In~\cite{DBLP:journals/appml/ChenSC08}, the authors made the following conjecture.

\begin{conj}
Let $G$ be a connected graph with $n(\geq 11)$ vertices and $\delta(G)\geq 3$, then $\gamma_{pr}(G)\leq 4n/7$.
\end{conj}

Goddard and Henning improved the result in Theorem~\ref{thm4}, and made a strengthened conjecture in~\cite{DBLP:journals/gc/GoddardH09}.
\begin{theorem}(\cite{DBLP:journals/gc/GoddardH09})
If $G$ is a connected cubic graph, then $\gamma_{pr}(G)\leq 3n/5$. Moreover, the equality holds if and only if $G$ is the Petersen Graph.
\end{theorem}

\begin{conj}
Let $G$ be a connected graph with $n$ vertices and $\delta(G)\geq 3$. If $G$ is not the Petersen Graph, then $\gamma_{pr}(G)\leq 4n/7$.
\end{conj}

Lu et al.~\cite{DBLP:journals/gc/LuWW16}  proved Conjecture 2 for $k$-regular graphs with $k\geq 4$, and later~\cite{DBLP:journals/dam/Lu0WW19} proved it for claw-free graphs with minimum degree at least 3.
\begin{theorem}(\cite{DBLP:journals/dam/Lu0WW19})
If G is a connected claw-free graph of order n with $\delta(G) \geq 3$, then $\gamma_{pr}(G) \leq 4n/7$.
\end{theorem}

People also give bound for the paired domination number of special graphs, such as $P_5$-free graphs~\cite{DBLP:journals/gc/DorbecG08}, subdivided star-free graphs~\cite{DBLP:journals/gc/DorbecG10}, generalized claw-free graphs~\cite{DBLP:journals/jco/DorbecGH07},  claw-free graphs~\cite{DBLP:journals/gc/HuangKS13}, and so on.

In this paper, we follow this line of research and prove Conjecture 2 for cubic graphs, which completes the picture for regular graphs. Our result does not rely on the graph to be claw-free. With the intuition that denser graph should have smaller power domination set, our result is a promising step towards proving Conjecture 1.
\section{Notations and Terminology}
Here we give a brief list of the graph theory concepts used in this paper, for other
notations and terminology, we refer readers to~\cite{bollobas2013modern}.

 All the graph considered in this paper are simple and undirected. The \textit{neighborhood} of a vertex $u\in V$ is $N_G(u) = \{v \in V|uv \in E\}$. The \textit{degree} of $v$ in G, denoted by $d_G(v)$, is the number $|N_G(v)|$. A \textit{cubic graph} is a graph in which every vertex has degree 3. We use $\delta(G)$ to denote the minimum vertex degree of $G$. A vertex with degree one is a \textit{leaf}. For a vertex set $S \subseteq V$, we use $G[S]$ to denote the subgraph induced by $S$. The neighborhood of $v$ in $S$ is $N_S(v) = N_G(v) \cap S$. For two disjoint vertex sets $X$ and $Y$, we use $[X,Y]$ to denote the set of edges between $X$ and $Y$.

Given a graph $G$, we say $S\subseteq V(G)$ is a paired dominating set(abbreviated as PDS) of $G$, if $S$ is a dominating set of $G$ and $G[S]$ has a perfect matching.
Let $S$ be a PDS of $G$, we say $x\in V\setminus S$ is a \textit{private neighbor} of $u\in S$, if $N_S(x)=\{u\}$. Fix a perfect matching $M$ of $G[S]$, we say $u$ and $\overline{u}$ form a pair if $u\overline{u}\in M$, and we use $S_u$ to denote the pair of vertices $u$ and $\overline{u}$.
A pair $S_u$ is a \textit{solo pair} if both $u$ and $\overline{u}$ are leaves in $G[S]$. Otherwise, we call it a \textit{linked pair}.
For a subset $S \subseteq V$, $\lambda(S)$ denotes the number of edges in $G[S]$. We use $u'$ to denote the private neighbor of $u\in S$ if $u$ has exactly one private neighbor. For any positive integer $n$, we use $[n]$ to denote the set $\{1,2,\ldots,n\}$.

Let $E_1$ be a subset of edges in $M$, the perfect matching of $G[S]$. Let $E_2$ be a subset of $E(G)$. \emph{Replacing $E_1$ with $E_2$} means to replace $S$ with $S'=S\setminus V(E_1)\cup V(E_2)$. Moreover, the edges in $E_2$ belong to the perfect matching of $G[S']$.

The Petersen Graph is a cubic graph with 10 vertices and 15 edges, see Figure~\ref{fig:Petersen} for an illustration.
\begin{figure}
\begin{center}
\begin{tikzpicture}[style=thick]
\draw [](18:2cm) -- (90:2cm) -- (162:2cm) -- (234:2cm) --
(306:2cm) -- cycle;
\draw (18:1cm) -- (162:1cm) -- (306:1cm) -- (90:1cm) --
(234:1cm) -- cycle;
\foreach \x in {18,90,162,234,306}{
\draw (\x:1cm) -- (\x:2cm);
\draw (\x:2cm)[fill]circle[radius=0.07];
\draw (\x:1cm)[fill]circle[radius=0.07];

}

\end{tikzpicture}
\end{center}
\caption{The Petersen Graph}
 \label{fig:Petersen}
\end{figure}
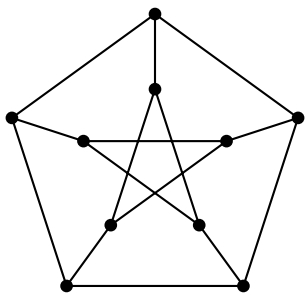

\section{Main Result}
In this section, we provide our main result, which is the following theorem.
\begin{theorem}\label{thm7}
Let $G$ be a cubic graph, then $\gamma_{pr}(G)\leq 4n/7$ if $G$ is not the Petersen Graph.
\end{theorem}
In the proof of Theorem~\ref{thm7}, we heavily use the following lemma.
\begin{lemma}(\cite{DBLP:journals/gc/HuangKS13})\label{lemma1}
Let $G$ be any connected graph and $S$ a minimum PDS of $G$. Suppose that $X\subseteq S$ and $Y\subseteq V\setminus  S$. If $S'=S\setminus X \cup Y$ dominates $X$ with
$|S'|<|S|$ and $G[S']$ has a perfect matching, then there exists a vertex $x$ in $V\setminus (S \cup Y )$ such that $N_S(x)\subseteq  X$.
\end{lemma}

Given a PDS $S$, and a perfect matching $M$ of $G[S]$, we partition $S$ into four subsets as follows:

\begin{enumerate}
\item $A = \{v \in S |$ $v$ together with $\overline{v}$ have at least two private neighbors$\}$,
\item $B = \{v \in S | v$ has a private neighbor and $\overline{v}$ has no private neighbor$\}$,
\item $C = \{v \in S | \overline{v} \in B\}$,
\item  $D = \{v \in S |$ neither $v$ nor $\overline{v}$ has any private neighbor$\}$.
\end{enumerate}

We call a pair of vertices in $S$ an \textbf{$A$-pair} if both of them belong to $A$; a \textbf{$BC$-pair} if one belongs to $B$ and the other belongs to $C$; and a \textbf{$D$-pair} if both belong to $D$.

Let $S$ be the minimum PDS of $G$ and $M$ a perfect matching of $G[S]$ that satisfies:
\begin{enumerate}[(P1)]
\item  $\lambda(S)$ is minimized;
\item  $|A\cup B|$  is minimized, subject to $P_1$.
\end{enumerate}

The following lemma lists some properties of $G[S]$ whose proof can be found in\cite{DBLP:journals/gc/HuangKS13}.
\begin{lemma}(\cite{DBLP:journals/gc/HuangKS13})\label{lemma2}
Let $G$ be a connected cubic graph and $S$ a minimum PDS of $G$ satisfying P1 and P2. Let $A, B, C$ and $D$ be the four vertex
sets defined accordingly. Then the following statements hold.
\begin{enumerate}
\item Each vertex in $C$ is a leaf in $G[S]$.
\item At least one vertex of each $D$ pair is a leaf in $G[S]$.
\item Let $u$ and $v$ be two different vertices in $B\cup D$, if $uv\in E$, then $G[S_u\cup S_v]$ induces a $P_4$ in $G[S]$. Moreover, there is a vertex $x\in V\setminus S$, such that $N_S(x)=\{\overline{u},\overline{v}\}$.
\end{enumerate}
\end{lemma}

We adopt the strategy used in~\cite{DBLP:journals/dam/Lu0WW19} to prove Conjecture 2 for cubic graphs. We design a weight function for $[S,V \setminus S]$, such that for every vertex in $V \setminus S$, the weights of all edges incident with it sum to one. Thus the total weight of edges in $[V \setminus S, S]$ are exactly $|V \setminus S|$.
At the same time, we show that averagely the edges incident with each pair have total weight at least 3/2. It follows that the total weight of edges in $[V\setminus S, S]$ is at least $3|S|/4$. Thus we have $3|S|/4 \leq |V \setminus S|$ and $|S| \leq 4n/7$.

We now define the weight function $f:[V\setminus  S,S]  \rightarrow [0,1]$.
\begin{enumerate}
\item If $x$ is a private neighbor of some vertex $u\in S$, then $f(xu)=1$;
\item Otherwise, if $N_S(x)\cap (B\cup C\cup D)\neq \emptyset$, then $f(xu)=0$, for any $u\in N_S(x)\cap A$;
\item If $N_S(x)$ is a subset of $A, B\cup C$ or $D$, then $f(xu)=1/|N_S(x)|$, for each $u\in N_S(x)$;
\item If $N_S(x)\cap D\neq \emptyset$, then $f(xu)=0$ for every vertex $u\in A$, $f(xu)=1/6$ for every vertex $u\in B\cup C$, and $f(xu)=\frac{1-|N_{B\cup C}(x)|/6}{|N_D(x)|}$ for every vertex $u\in D$.
\end{enumerate}

According to the definition of the weight function $f$, we have the following observations.
\begin{lemma}
For any edge $xu$ with $x\in V\setminus  S$, $u\in S$, the following statements hold.
\begin{enumerate}
\item $f(xu)\in \{0,1/6,1/3,5/12,1/2,2/3,5/6,1\}$;
\item $\sum_{u\in N_S(x)} f(xu)=1$;
\item If $u\in B\cup C$, then $f(xu)\in \{1/6,1/3,1/2,1\}$;
\item If $u\in D$, then $f(xu)\in \{1/3,5/12,1/2,2/3,5/6,1\}$. Moreover, $f(xu)=1/3$ if and only if $|N_S(x)|=3$ and $N_S(x)\subseteq D$.
\end{enumerate}
\end{lemma}

For any $S'\subseteq S$, we use $f(S')$ to denote the total weights of edges  in $[S',V\setminus  S]$. If $S'=\{u\}$, then we abbreviate $f(S')$ as $f(u)$.


For each  component in $G[S]$, we prove that it has enough total weight. We show that if it does not have enough total weight, then there is a smaller PDS, or a PDS with smaller $\lambda$ or smaller $A\cup B$, which contradicts the choice of $S$. 

The following lemma deals with components in $G[S]$ that contain no $D$ pair.  Note that every component in $S$ contains at most two $BC$ pairs, as a $BC$ pair can be adjacent with at most one pair in $G[S]$. An $A$ pair can be adjacent with at most two pairs in $G[S]$ as it has at least two private neighbors.


\begin{lemma}
Let $\mathcal{C}$ be a connected component in $G[S]$ that contains no $D$ pair, then $f(V(\mathcal{C}))\geq 3|V(\mathcal{C})|/4$.
\end{lemma}
\begin{proof}
If $\mathcal{C}=S_u$ is a solo $A$ pair, then $f(S_u)\geq 2$, because $S_u$ have at least 2 private neighbors.

If $\mathcal{C}=S_u$ is a solo $BC$ pair, then $f(S_u)\geq 1+ 1/6*3=3/2$, because $u$ or $\overline{u}$ has a private neighbor, and each edge in $[B\cup C,V\setminus S]$ has weight at least 1/6.

If $\mathcal{C}$ consists of two $BC$ pair $S_u$ and $S_v$ with $uv\in E(G)$, then according to Lemma~\ref{lemma2}, there is a vertex $x\in V\setminus S$, such that $N_S(x)=\{\overline{u}, \overline{v}\}$. Thus $f(S_u)+f(S_v)\geq f(u')+f(v')+f(x)=3$.

If $\mathcal{C}$ consists of $a(\geq 1)$ $A$ pairs and $b(\leq 2)$ $BC$ pairs, then the total weight of it is at least $2a+b+b/3\geq 3(a+b)/2$. Here we use the fact that each vertex in $C$ has two edges to $V\setminus S$, which totally have weight 1/3.

\end{proof}

Now we consider the weight of components containing at least one $D$ pair. A $D$ pair adjacent with another pair in $S$ is called a \textit{linked $D$ pair}. Two adjacent $D$ pairs are called \textit{linked $D$ pairs}.  

\subsection{Component with linked $D$ Pair}
\subsubsection{Component with No Linked  D pairs}

In this section, we discuss the weight of $\mathcal{C}$ if it contains no linked $D$ pairs. First we give some observations of $\mathcal{C}$.


\begin{lemma}\label{lemma11}
{Let $\mathcal{C}$ be a connected component in $G[S]$ that contains no linked $D$ pairs. An $A$ pair in $\mathcal{C}$ that contains a vertex with no private neighbor is a solo pair.}
\end{lemma}
\begin{proof}
If $S_u$ is an $A$ pair in which $u$ has no private neighbor, then $\overline{u}$ has at least two private neighbors. Let $\overline{u}'$ be a private neighbor of $\overline{u}$. If $S_u$ is not a solo pair, then $d_S(u)\geq 2$. It follows that replacing $\{u\overline{u}\}$ with $\{\overline{u}\overline{u}'\}$ gives a PDS with smaller $\lambda$, and so $S_u$ is a solo pair.
\end{proof}

\begin{lemma}\label{lemma9}
Let $\mathcal{C}$ be a connected component in $G[S]$ that contains no linked $D$ pairs. There is no $D$ pair $S_u$ in $\mathcal{C}$ such that $d_A(u)=d_B(u)=d_D(u)=1$.
\end{lemma}
\begin{proof}
Let $S_u$ be a $D$ pair such that $d_A(u)=d_B(u)=d_D(u)=1$. Suppose $uv, uw\in E(G)$, where $v\in B$ and $w\in A$. Then according to Lemma~\ref{lemma2}, there is a vertex $x\in V\setminus S$, such that $N_S(x)=\{\overline{u},\overline{v}\}$. Replacing $\{u\overline{u}, v\overline{v}\}$ with $\{vv', x\overline{u}\}$ gives a PDS with smaller $\lambda$, a contradiction.
\end{proof}
According to Lemma~\ref{lemma9}, if $\mathcal{C}$ contains a $BC$ pair $S_v$, then there is no $A$ pair in $\mathcal{C}$, as $\mathcal{C}$ contains no liked $D$ pairs. 

\begin{figure}
\centering
\includegraphics[scale=0.5]{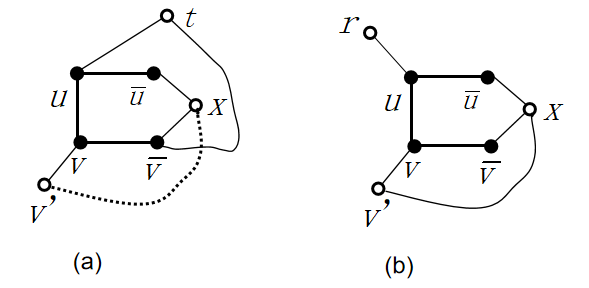}
\caption{C contains only a $D$ pair and a $BC$ pair}
 \label{fig:2}
\end{figure}

Suppose $\mathcal{C}$ only contains $S_u$ and $S_v$, where $uv\in E(G)$ and $v\in B$. See Figure~\ref{fig:2} for an illustration. 
According to Lemma~\ref{lemma2}, there is a vertex $x\in V\setminus  S$, such that $N_s(x)=\{\overline{u},\overline{v}\}$. If $xv'\not\in E(G)$, then replacing $\{u\overline{u}, v\overline{v}\}$ with $\{x\overline{u}, vv'\}$ gives a PDS with smaller $\lambda$, or there is a vertex $t\in V\setminus S$, such that $N_S(t)=\{u,\overline{v}\}$. In the later case, $f(S_u)+f(S_v)> f(v')+ f(x) +f(t)=3$. If $xv'\in E(G)$, then let $r\in V\setminus S$ be the last neighbor of $u$.    If  $r$ has a neighbor $w$ in $D\setminus S_u$. Then replacing $\{u\overline{u}, v\overline{v}, w\overline{w}\}$ with $\{rw, x\overline{v}\}$ gives a smaller PDS, or there is a vertex $t\in V\setminus S$, such that $N_S(t)=\{\overline{u},\overline{w}\}$. In the later case,  $f(S_u)+f(S_v)\geq f(x)+f(v')+f(ur)+f(\overline{u}t)+f(\overline{v})\geq 2+ 1/2 +1/6+ 1/3=3$.
 Otherwise, $N_D(r)\subseteq S_u$. Then $f(S_u)+f(S_v)\geq 3$, as $f(ru)\geq 2/3$ if $r\overline{u}\not\in E(G)$, and $f(ru)=f(r\overline{u})\geq 5/12$ otherwise.

Suppose $\mathcal{C}$ contains a $D$ pair $S_u$ and two $BC$ pairs $S_v$ and $S_w$, with $v,w\in B$ and $uv,uw\in E(G)$.
According to Lemma~\ref{lemma2}, there are two vertices $x,y\in V\setminus  S$, such that $N_s(x)=\{\overline{u},\overline{v}\}$, $N_s(y)=\{\overline{u},\overline{w}\}$. Then replacing $\{u\overline{u},v\overline{v}\}$ with $\{vv', \overline{u}x\}$ gives a PDS with smaller $\lambda$.

Now consider the case when $\mathcal{C}$ contains no $BC$ pair.
\begin{lemma}\label{lemma10}
Let $\mathcal{C}$ be a connected component in $G[S]$ that contains no linked $D$ pairs and no $BC$ pair.  Let $d$ be the number of $D$ pairs in $\mathcal{C}$, then $d\leq 2$.
\end{lemma}
\begin{proof}
Note that $\mathcal{C}$ contains no edge between $D$ pairs.
If $d\geq 3$, then $\mathcal{C}$ contains three $D$ pairs $S_u, S_v, S_w$, such that there are two $A$ pairs $S_p, S_q$ with $\{up, \overline{p}\overline{v}, \overline{q}\overline{v}, qw\}\subseteq E(G)$ and every vertex in  $\{p,\overline{p},q,\overline{q}\}$ has at least one private neighbor, according to Lemma~\ref{lemma11}.  Then replacing $\{u\overline{u}, p\overline{p}, v\overline{v}\}$ with $\{up,\overline{p}\overline{v}\}$ gives a smaller PDS, or there is a vertex $t_1\in V\setminus S$, such that $N_S(t_1)=\{\overline{u},v\}$, according to Lemma~\ref{lemma2}. Similarly, there is a vertex $t_2\in V\setminus S$, such that $N_S(t_2)=\{v,\overline{w}\}$. Then replacing $\{u\overline{u},v\overline{v}\}$ with $\{t_1\overline{u}\}$ gives a smaller PDS, a contradiction. Therefore, the number of $D$ pairs in $\mathcal{C}$ is at most 2.
\end{proof}
Let $d$ be the number of $D$ pairs in $\mathcal{C}$, in which there is no $BC$ pair. Then according to Lemma~\ref{lemma10}, $d\leq 2$ and there are at least $d-1$ $A$ pairs in $\mathcal{C}$. 

If $d=1$, then $\mathcal{C}$ contains $S_u$ and an $A$ pair $S_v$, because we are considering component that contains a linked $D$ pair. Then the three edges in $[S_u,V\setminus S]$ have total weight at least 3*1/3=1. Thus $f(S_u)+f(S_v)\geq 3$,  as $S_v$ has at least two private neighbors. 

If $d=2$, then $S_u$ is adjacent with an A pair $S_w$, which is adjacent with another $D$ pair $S_v$ in $\mathcal{C}$, where $uw\in E(G[S])$ and $\overline{v}\overline{w} \in E(G[S])$. Then replacing $\{u\overline{u},v\overline{v},w\overline{w}\}$ with $\{\overline{v}\overline{w},uw\}$ gives a smaller PDS, or there is a vertex $t_1\in V\setminus  S$, such that $N_S(t_1)= \{\overline{u},v\}$. In the later case, replacing $\{u\overline{u},v\overline{v}\}$ with $\{t_1\overline{u} \}$ gives a smaller PDS, or there is a vertex $t_2\in V\setminus  S$, such that $N_S(t_2)\subseteq  \{u,v,\overline{v}\}$. In the later case, replacing $\{u\overline{u},v\overline{v}\}$ with $\{t_1v \}$ gives a smaller PDS, or there is a vertex $t_3\in V\setminus  S$, such that $N_S(t_3)\subseteq  \{u,\overline{u},\overline{v}\}$. In the later case, $f(S_u)+f(S_v)+f(S_w)\geq f(w')+f(\overline{w}')+f(t_1)+f(t_2)+f(t_3)=5$.

\subsubsection{Component with Linked D pairs}

 If $S_u$ is adjacent with two D pairs $S_v$ and $S_w$, where $uv,uw\in E(G)$, then according to Lemma~\ref{lemma2}, there are two vertices $x, y\in V\setminus  S$, such that $N_S(x)=\{\overline{u},\overline{v}\}$, $N_S(y)=\{\overline{u},\overline{w}\}$.  In this case, replacing
$\{u\overline{u}, w\overline{w}\}$ with $\{\overline{w}y\}$ gives a smaller PDS. 

If $S_u$ is adjacent with a $D$ pair $S_v$ and an $A$ pair $S_w$, where $uv,uw\in E(G)$, then  according to Lemma~\ref{lemma2}, there is a vertex $x\in V\setminus  S$, such that $N_S(x)=\{\overline{u},\overline{v}\}$.
In this case, replacing $\{u\overline{u}\}$ with $\{x\overline{u}\}$ gives a PDS with smaller $\lambda$. 

If $S_u$ is adjacent with a $D$ pair $S_v$ and a $BC$ pair $S_w$, where $uv,uw\in E(G)$ and $w\in B$, then according to Lemma~\ref{lemma2}, there are two vertices $x, y\in V\setminus  S$, such that $N_S(x)=\{\overline{u},\overline{v}\}$, $N_S(y)=\{\overline{u},\overline{w}\}$. Then replacing $\{u\overline{u},w\overline{w}\}$ with $\{ww',y\overline{u}\}$ gives a PDS with smaller $\lambda$.

Now we only need to consider the case when $S_u$ is only adjacent with one $D$ pair $S_v$.
According to Lemma~\ref{lemma2}, there is a vertex $x\in V\setminus  S$, such that $N_S(x)=\{\overline{u},\overline{v}\}$. Besides $x\overline{u}$ and $x\overline{v}$, there are four other edges in $[S_u\cup S_v,V\setminus  S]$. To have $f(S_u)+f(S_v)\geq 3$, we need these four edges to have total weight at least 2. Note that for any vertex $x\in V\setminus  S$, if $u\in D$, then $f(xu)\in \{1/3,5/12,1/2,2/3,5/6,1\}$. Moreover, $f(xu)=1/3$ if and only if $|N_D(x)|=3$.

\begin{lemma}\label{lemma5}
Let $S_u$ and $S_v$ be two adjacent $D$ pairs such that $uv\in E(G)$. If $\overline{u}$ has a neighbor $t\in V\setminus S$, which has a neighbor $w\in B\cup D\setminus (S_u\cup S_v)$, then there is a vertex $t_1\in V\setminus  S$, such that $N_S(t_1)=\{u,\overline{w}\}$.
\end{lemma}
\begin{proof}
According to Lemma~\ref{lemma1}, replacing $\{u\overline{u},w\overline{w}\}$ with $\{tw\}$ gives a smaller PDS, or there is a vertex $t_1\in V\setminus  S$, such that $N_S(t_1)=\{u,\overline{w}\}$.
\end{proof}
\begin{lemma}\label{lemma6}
Let $S_u$ and $S_v$ be two adjacent $D$ pairs such that $uv\in E(G)$. If $u$ has a neighbor $r\in V\setminus  S$ which has a neighbor $w\in B\cup D \setminus (S_u\cup S_v)$, then there is a vertex $t_1\in V\setminus  S$, such that $N_S(t_1)\subseteq \{\overline{u},v,\overline{w}\}$.
\end{lemma}

\begin{proof}
According to Lemma~\ref{lemma1}, replacing $\{ u\overline{u},w\overline{w},v\overline{v}\}$ with $\{rw,x\overline{v}\}$ gives a smaller PDS, or there is a vertex vertex $t_1\in V\setminus  S$, such that $N_S(t_1)\subseteq \{\overline{u},v,\overline{w}\}$.
\end{proof}

 If the total weight of $S_u$ and $S_v$ is less than 3, then at least one edge in $[S_u\cup S_v,V\setminus S]$ should have weight  1/3 or 5/12. We discuss all possible cases accordingly. Let $t$ be the neighbor of $\overline{u}$ in $V\setminus S$, and $r$ the neighbor of $u$ in $V\setminus S$.  Symmetrically, we just need to consider the cases when $f(\overline{u}t),f(ur)\in \{1/3,5/12\}$.

\begin{figure}
\centering
\includegraphics[scale=0.5]{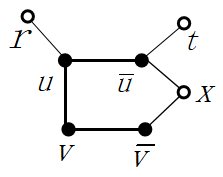}
\caption{Two linked $D$ pairs}
 \label{fig:3}
\end{figure}

\textbf{Case 1:} $f(\overline{u}t)=1/3$.

\textbf{Subcase 1.1:}  All three neighbors of $t$ are in $S_u\cup S_v$.

If $N_S(t)=\{ \overline{u},v,\overline{v} \}$, then replacing $\{u\overline{u},v\overline{v}\}$ with $\{t\overline{u}\}$ gives a smaller PDS.

If $N_S(t)=\{ u,\overline{u},\overline{v} \}$, then replacing $\{u\overline{u},v\overline{v}\}$ with $\{t\overline{v}\}$ gives a smaller PDS.

If $N_S(t)=\{ u,\overline{u},v \}$, then replacing $\{u\overline{u},v\overline{v}\}$ with $\{ut,x\overline{v}\}$ gives a PDS with smaller $\lambda$.

\textbf{Subcase 1.2:} Two neighbors of $t$ are in $S_u \cup S_v$.

Then $t$ has a $D$ neighbor $w\not\in S_u\cup S_v$.

If $N_S(t)=\{ u,\overline{u},w \}$, then replacing $\{u\overline{u},w\overline{w}\}$ with $\{tw\}$ gives a smaller PDS.

If $N_S(t)=\{ \overline{u},\overline{v},w \}$, then according to Lemma~\ref{lemma5}, there are two vertices $t_1,t_2\in V\setminus  S$, such that $N_S(t_1)=\{u,\overline{w}\}$, $N_S(t_2)=\{v,\overline{w}\}$. In this case, replacing $\{u\overline{u},v\overline{v},w\overline{w}\}$ with $\{x\overline{u},t_2\overline{w}\}$ gives a smaller PDS.

If $N_S(t)=\{ \overline{u},v,w \}$, then according to Lemma~\ref{lemma5}, there is a vertex $t_1\in V\setminus S$, such that $N_S(t_1)=\{u,\overline{w}\}$. If $xt_1\not\in E(G)$, then replacing $\{u\overline{u},v\overline{v},w\overline{w}\}$ with $\{ut_1,x\overline{v},tw\}$ gives a PDS with smaller $\lambda$. Otherwise, replacing $\{u\overline{u},v\overline{v},w\overline{w}\}$ with $\{xt_1, tw\} $ gives a smaller PDS, or there is a vertex $t_2\in V\setminus S$, such that $N_S(t_2)=\{\overline{v},\overline{w}\}$. In the later case, replacing $\{u\overline{u},v\overline{v},w\overline{w}\}$ with $\{\overline{w}t_2, t\overline{u}\} $ gives a smaller PDS. 

\textbf{Subcase 1.3:} One neighbor of $t$ is in $S_u\cup S_v$.

Let $w_1, w_2$ be the other two $D$ neighbors of $t$. Then there is an $i\in [2]$, such that there is no vertex $y\in V\setminus  S$, satisfying $N_S(y)=\{u,\overline{w_i}\}$, because $u$ already has degree 2. Thus replacing $\{u\overline{u}, w\overline{w}\}$ by $\{tw_i\}$ gives a smaller PDS.

\textbf{Case 2:} $f(ur)=1/3$.

\textbf{Subcase 2.1:} Three neighbors of $r$ are in $S_u\cup S_v$.


In this case, $r$ must be adjacent with $\overline{u}$ or $\overline{v}$, which we already discussed in Case 1, as $f(\overline{u}t)=1/3$ or $f(\overline{v}t)=1/3$.

\textbf{Subcase 2.2:} Two neighbors of $r$ are in $S_u\cup S_v$.

We only need to consider the case when $N_S(r)=\{u,v,w\}$, where $w\in D\setminus (S_u\cup S_v)$. Replacing $\{u\overline{u}, v\overline{v},w\overline{w}\}$ with $\{wr,x\overline{v}\}$ gives a smaller PDS, or there is a vertex $t_1\in V\setminus  S$, such that $N_S(t_1)=\{\overline{w},\overline{u}\}$. In the later case, replacing $\{u\overline{u},w\overline{w}\}$ with $\{t_1\overline{w}\}$ gives a smaller PDS.

\textbf{Subcase 2.3:} One neighbor of $r$ is in $S_u\cup S_v$.

Let $w_1,w_2$ be the other two $D$ neighbors of $r$. According to Lemma~\ref{lemma6}, there are two vertex $t_1,t_2\in V\setminus  S$, such that $N_S(t_1)\subseteq \{\overline{w_1}, \overline{u},v \}$, $N_S(t_2)\subseteq \{\overline{w_2}, \overline{u},v \}$.

If $N_S(t_1)= \{\overline{w_1}, \overline{u},v \}$, then $f(\overline{u}t_1)=1/3$, which we already discussed in Subcase 1.2.

If $N_S(t_1)= \{\overline{w_1}, \overline{u}\}$, then replacing $\{ w_1\overline{w_1},u\overline{u} \}$ with $\{t_1\overline{w_1}\}$ gives a smaller PDS.

If $N_S(t_1)= \{\overline{w_1}, v \}$, then $N_S(t_2)= \{\overline{w_2}, \overline{u} \}$. This is exactly the same with the above case.

If $N_S(t_1)= \{ \overline{u},v \}$, then $t_1$ and $x$ already provide weight 2 for $S_u\cup S_v$. Consider the neighbor $p$ of $\overline{v}$. If $f(p\overline{v})\geq 2/3$, then $f(S_u)+f(S_v)\geq 3$. Thus we just need to consider the case when $p$ has at least two $D$ neighbors. Let $w_3\neq \overline{v}$ be a $D$ neighbor of $p$. Then replacing $\{v\overline{v}, w_3\overline{w_3} \}$ with $\{pw_3\}$ gives a smaller PDS.

\textbf{Case 3:} $f(\overline{u}t)=5/12$.

Note that in this case, $t$ has two neighbors in $D$.

\textbf{Subcase 3.1:} Two $D$ neighbors of $t$ are in $S_u\cup S_v$.

Let $w$ be the other neighbor of $t$, which belongs to $B\cup C$.

If $N_D(t)=\{u,\overline{u}\}$, and $w\in B$, then replacing $\{u\overline{u},w\overline{w} \}$ with $\{tw \}$ gives a smaller PDS.

If $N_D(t)=\{u,\overline{u}\}$, and $w\in C$, then replacing $\{u\overline{u},w\overline{w}\}$ with $\{t\overline{u}, \overline{w}\overline{w}'\}$ gives a PDS with smaller $\lambda$.

If $N_D(t)=\{v,\overline{u}\}$, and $w\in B$, then replacing $\{u\overline{u}, w\overline{w} \}$ with $\{tw\}$ give a smaller PDS, or there is a vertex $t_1\in V\setminus  S$, such that $N_S(t_1)=\{u,\overline{w} \}$. In the later case, if $xt_1\not\in E(G)$, then replacing $\{ u\overline{u},v\overline{v},w\overline{w}\}$ with $\{x\overline{v}, ut_1,tw\}$ gives a PDS with smaller $\lambda$. Otherwise, replacing $\{ u\overline{u},v\overline{v},w\overline{w}\}$ with $\{xt_1,tw \}$ gives a smaller PDS, or there is a vertex $t_2\in V\setminus  S$, such that $N_S(t_2)=\{\overline{v},\overline{w}\}$. In the later case, replacing $\{ u\overline{u},v\overline{v},w\overline{w}\}$ with $\{ut_1, tw, \overline{v}t_2\}$ gives a PDS with smaller $\lambda$.

If $N_S(t)=\{v,\overline{u}\}$ and $w\in C$, then suppose $p\in V\setminus S$ is the last neighbor of $\overline{v}$, such that $f(\overline{v}p)=5/12$. If $p$ has a $D$ neighbor $w_2\neq u$, then replacing $\{v\overline{v},w_2\overline{w_2}\} $ with $\{pw_2\}$ gives a smaller PDS.

Otherwise, $N_S(p)=\{u,\overline{v},w_2\}$ and $w_2\in C$. If $w_2=w$, then replacing $\{ u\overline{u}, v\overline{v}, w\overline{w}\}$ with $\{\overline{u}t, \overline{v}p, \overline{w}\overline{w}' \}$ gives a PDS with smaller $\lambda$. If $w_2\neq w$ and $\overline{w}'\overline{w_2}'\not\in E(G)$, then replacing $\{u\overline{u}, v\overline{v}, w\overline{w}, w_2\overline{w_2} \}$ with $\{\overline{u}t, \overline{v}p, \overline{w}\overline{w}', \overline{w_2}\overline{w_2}' \}$ gives a PDS with smaller $\lambda$, or there is a vertex $t_3\in V\setminus  S$, such that $N_S(t_3)=\{w,w_2 \}$. In the later case, $f(S_u)+f(S_v)+f(S_{w})+f(S_{w_2})\geq 6$.  If $w_2\neq w$ and $w'\overline{w_2}'\in E(G)$, then  replacing $\{u\overline{u}, v\overline{v}, w\overline{w}, w_2\overline{w_2} \}$ with $\{\overline{u}t, p\overline{w_2}, \overline{w}\overline{w}'\}$ gives a smaller PDS, or there is a vertex $t_1\in V\setminus S$, such that $N_S(t_1)=\{w,\overline{w_2}\}$. In the later case,  $f(S_u)+f(S_v)+f(S_{w})+f(S_{w_2})\geq f(x)+f(t)+f(p)+f(\overline{w}')+f(\overline{w_2}')+f(t_1)= 6$.   {It is safe to compute the weight of these four pairs together, as all vertices in $S_u\cup S_v\cup S_w\cup S_{w_2}$, except $\overline{w}$ and $\overline{w_2}$, have degree 3.}


If $N_D(t)=\{\overline{u},\overline{v}\}$ and $w\in B$, then  replacing $\{u\overline{u}, w\overline{w} \}$ with $\{tw\}$ gives a smaller PDS, or there is a vertex $t_1\in V\setminus  S$, such that $N_S(t_1)=\{u,\overline{w} \}$. In the later case, replacing $\{v\overline{v}, w\overline{w} \}$ with $\{tw\}$ gives a smaller PDS, or there is a vertex $t_2\in V\setminus  S$, such that $N_S(t_2)=\{v,\overline{w} \}$. In the later case, $f(S_u\cup S_v)\geq f(x)+f(\overline{u}t)+f(\overline{v}t)+f(ut_1)+f(vt_2)=1+3*5/6>3$. 

If $N_D(t)=\{\overline{u},\overline{v}\}$ and $w\in C$, then $S_u$ and $S_v$ already have weight 1+5/6. Let $N(u)=\{\overline{u},v,r\}$, $N(v)=\{u,\overline{v},p\}$. If both $r$ and $p$ have exactly one $D$ neighbor, then $f(ur)\geq 2/3$ and $f(vp)\geq 2/3$. Thus $f(S_u)+f(S_v)\geq 1+5/6+4/3>3$.  Otherwise, $r$ or $p$ has another $D$ neighbor.

If $r$ has a $D$ neighbor $w_1\not\in S_u\cup S_v$. Then according to Lemma~\ref{lemma6}, there is a vertex $t_2\in V\setminus S$, such that $N_S(t_2)\subseteq\{\overline{u},v,\overline{w_1}\}$. As $\overline{u}$ already has degree 3, we have $N_S(t_2)=\{v,\overline{w_1}\}$. In this case, replacing $\{u\overline{u}, v\overline{v}, w_1\overline{w_1}\}$ with $\{t_2\overline{w_1}, x\overline{u}\}$ gives a smaller PDS. The same argument works for the case when $p$ has a $D$ neighbor $w_1\not\in S_u\cup S_v$.

Now we need to consider the case when both $r$ and $p$ have no $D$ neighbor outside of $S_u\cup S_v$. This is only possible when $r=p$. Let $w_1$ be the other neighbor of $r\in S$. If $w_1\in B$, then replacing $\{u\overline{u},v\overline{v},w_1\overline{w_1}\}$ with $\{rw_1, x\overline{u}\}$ gives a smaller PDS. If $w_1\in C$ and $x\overline{w_1}'\not\in E(G)$, then replacing $\{u\overline{u},v\overline{v},w_1\overline{w_1}\}$ with $\{vr, x\overline{u}, \overline{w_1}\overline{w_1}'\}$ gives a PDS with smaller $\lambda$. Otherwise, $w_1\in C$ and $x\overline{w_1}\in E(G)$, then  replacing $\{ u\overline{u}, v\overline{v}, w_1\overline{w_1}\}$ with $\{x\overline{u},rw_1 \}$ gives a smaller PDS. 


\textbf{Subcase 3.2:} One $D$ neighbor of $t$ is in $S_u\cup S_v$.

Let $w\not\in S_u\cup S_v$ be the other $D$ neighbor of $t$. According to Lemma~\ref{lemma5}, there is a vertex $t_1\in V\setminus  S$, such that $N_S(t_1)=\{u,\overline{w}\}$. If $tt_1\in E(G)$, then replacing $\{u\overline{u},w\overline{w}\}$ with $\{ tw\}$ gives a smaller PDS. 
Otherwise, replacing $\{ u\overline{u},w\overline{w}\}$ with $\{t\overline{u},t_1\overline{w}\}$ gives a PDS with smaller $\lambda$. 

\textbf{Case 4:} $f(ur)=5/12$.

Note that in this case, $r$ has two neighbors in $D$, one neighbor $w$ in $B\cup C$.

\textbf{Subcase 4.1:} Two $D$ neighbors of $r$ are in $S_u\cup S_v$.

As we already considered the case when the edges in $[\overline{u},V\setminus S]$ have weight 5/12, we only need to consider the case when $N_D(r)=\{u,v\}$.

If $w\in B$, then according to Lemma~\ref{lemma5}, there are two vertices $t_1,t_2\in V\setminus  S$, such that $N_S(t_1)=\{\overline{u},\overline{w}\}$, $N_S(t_2)=\{\overline{v},\overline{w}\}$. If $w't_{i}\not\in E(G)$, for any $i\in [2]$, then replacing $\{u\overline{u}, w\overline{w}  \}$ with $\{ww',t_{i}\overline{u}\}$ gives a PDS with smaller $\lambda$. Otherwise, $w't_1,w't_2\in E(G)$. Then replacing $\{u\overline{u},v\overline{v},w\overline{w}\}$ with $\{rv, x\overline{u},t_2w'\}$ gives a PDS with smaller $\lambda$. 

If $w\in C$ and $\overline{w}'x\not\in E(G)$, then replacing $\{ u\overline{u},v\overline{v}, w\overline{w} \}$ with $\{\overline{w}\overline{w}',rv,\overline{u}x\}$  gives a PDS with smaller $\lambda$, or there is a vertex $t_1\in V\setminus  S$, such that $N_S(t_1)=\{w,\overline{v}\}$. In the later case, replacing $\{ u\overline{u},v\overline{v}, w\overline{w} \}$ with $\{\overline{w}\overline{w}',x\overline{v},ru\}$ gives a PDS with smaller $\lambda$. 

If $w\in C$ and $\overline{w}'x\in E(G)$, then replacing $\{ u\overline{u},v\overline{v}, w\overline{w} \}$ with $\{rw,\overline{u}x\}$  gives a smaller PDS, or there is a vertex $t_1\in V\setminus S$, such that $N_S(t_1)=\{\overline{v},\overline{w} \}$. In the later case,  replacing $\{ u\overline{u},v\overline{v}, w\overline{w} \}$ with $\{rw,\overline{v}x\}$  gives a smaller PDS.

\textbf{Subcase 4.2:} One $D$ neighbor of $r$ is in $S_u\cup S_v$.

Let $w\not\in S_u\cup S_v$ be the other $D$ neighbor of $r$. According to Lemma~\ref{lemma5}, there is a vertex $t_1\in V\setminus  S$, such that $N_S(t_1)\subseteq \{ \overline{u}, v, \overline{w}\}$. As we already considered the case when the edge in $[\overline{u},V\setminus S]$ has weight 1/3. We only need to consider the cases when $N_S(t_1)=\{ \overline{u}, \overline{w} \}$, $N_S(t_1)=\{ v, \overline{w}\}$ or $N_S(t_1)=\{ \overline{u},v\}$.

If $N_S(t_1)=\{ \overline{u}, \overline{w} \}$, then replacing $\{u\overline{u},w\overline{w} \}$ with $\{rw, \overline{u}t_1 \}$ gives a PDS with smaller $\lambda$. 


If $N_S(t_1)=\{ v, \overline{w}\}$, then replacing $\{u\overline{u},v\overline{v},w\overline{w}\}$ with $\{\overline{w}t_1,x\overline{u} \}$ gives a smaller PDS, or there is a vertex $t_2\in V\setminus S$, such that $N_S(t_2)=\{w,\overline{v}\}$. In the later case, $S_u$ and $S_v$ get the weight of  $ur,vt_1, \overline{v}t_2$ and $x$, which sums to 2+5/12. Let $t$ be the neighbor of $\overline{u}$. As we already considered the case when $f(\overline{u}t)\in \{1/3,5/12\}$, we only need consider the case when $f(\overline{v}p)>5/12$, i.e. $f(\overline{v}p)\geq 1/2$. If $f(\overline{u}t)=1/2$, then $t$ has another $D$ neighbor $w_2$. Then replacing $\{u\overline{u},w_2\overline{w_2} \}$ with $\{tw_2\} $ gives a smaller PDS. Thus, $f(\overline{u}t)\geq 2/3$, and so $f(S_u)+f(S_v)\geq 5/12+2+2/3>3$.


If $N_S(t_1)=\{ \overline{u},v\}$, then let $p\in V\setminus S$ be the last neighbor of $\overline{v}$. As we already considered the case when $f(\overline{v}p)\in \{1/3,5/12\}$, we only need consider the case when $f(\overline{v}p)>5/12$, i.e. $f(\overline{v}p)\geq 1/2$. If $f(\overline{v}p)=1/2$, then $p$ has another $D$ neighbor $w_2$. Then replacing $\{v\overline{v},w_2\overline{w_2} \}$ with $\{pw_2\} $ gives a smaller PDS. Thus, $f(\overline{v}p)\geq 2/3$, and so $f(S_u)+f(S_v)\geq 5/12+2+2/3>3$. 

\subsection{Solo $D$ Pair}
Let $S_u$ be a solo $D$ pair. There are four edges in $[S_u, V\setminus S]$.  Note that for every edge $xu$ with $x\in V\setminus  S$, $u\in D$, we have $f(xu)\in\{ 1/3,5/12,1/2,2/3,5/6,1\}$.
Because $1/3+5/12=3/4$ and $1/3*3+1/2=3/2$, there are only two possible cases for $f(S_u)$ to be less than $3/2$. Either all the four edges have weight $1/3$, or three edges have weight $1/3$ and one edge have weight 5/12. We show both cases are not possible.

We first discuss the case when $u$ and $\overline{u}$ have a common neighbor in $V\setminus  S$.

\textbf{Case 1:} $u$ and $\overline{u}$ have a common neighbor $x\in V\setminus  S$.


If there are at least 2 edges with weight at least 5/12 incident with $S_u$, then $f(S_u)\geq 3/2$. Thus, we only need to discuss the case when $|N_D(x)|=3$.

Suppose $|N_D(x)|=3$, and $v\in D$ is the other neighbor of $x$. Then replacing $\{u\overline{u},v\overline{v}\}$ with $\{xv\}$ gives a smaller PDS, or there is a vertex $y\in V\setminus  S$, such that $N_S(y)\subseteq \{u,\overline{u},\overline{v}\}$. For the same reason, we only need to consider the case that $N_S(y)= \{u,\overline{u},\overline{v}\}$.

If $v$(or $\overline{v}$) has degree at least 2 in $G[S]$, then replacing $\{u\overline{u},v\overline{v}\}$ with $\{\overline{u}y\}$(or $\{ xu\}$, respectively) gives a smaller PDS. Thus
we only need consider the case when $S_v$ is also a solo $D$ pair. Now we discuss whether $v$ and $\overline{v}$ have a common neighbor in $V\setminus  S$.


\textbf{Subcase 1.1:} $v$ and $\overline{v}$ have a common neighbor $z\in V\setminus  S$, see Figure~\ref{fig:linkedD} for an illustration.

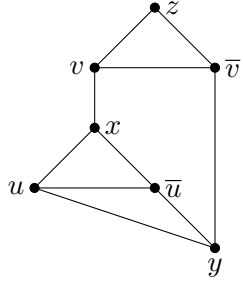
\begin{figure}
\centering
\begin{tikzpicture}[scale = 0.4]

\draw [](2,0)node[below]{$y$};
\draw [](-4,2)node[left]{$u$};
\draw [](0,2)node[right]{$\overline{u}$};
\draw [](-2,4)node[right]{$x$};
\draw [](-2,6)node[left]{$v$};
\draw [](2,6)node[right]{$\overline{v}$};
\draw [](0,8)node[right]{$z$};

\draw [](2,0)[fill]circle[radius=0.15];
\draw [](-4,2)[fill]circle[radius=0.15];
\draw [](0,2)[fill]circle[radius=0.15];
\draw [](-2,4)[fill]circle[radius=0.15];
\draw [](-2,6)[fill]circle[radius=0.15];
\draw [](2,6)[fill]circle[radius=0.15];
\draw [](0,8)[fill]circle[radius=0.15];

\draw [](2,0)--(-4,2);
\draw [](2,0)--(0,2);
\draw [](2,0)--(2,6);
\draw [](-4,2)--(0,2);
\draw [](-4,2)--(-2,4);
\draw [](-2,4)--(-2,6);
\draw [](-2,6)--(2,6);
\draw [](-2,6)--(0,8);
\draw [](2,6)--(0,8);
\draw [](-2,4)--(0,2);
\end{tikzpicture}
\caption{Subcase 1.1}
 \label{fig:linkedD}
\end{figure}

If $N_S(z)=\{v,\overline{v}\}$ or $N_S(z)\cap A\neq \emptyset$, then $f(S_u)+f(S_v)=3$ and we are done. Otherwise, suppose $d_S(z)=3$ and $w\in S\setminus A$ is the third neighbor of $z$. If $w\not\in C$, then replacing $\{w\overline{w}, v\overline{v}, u\overline{u} \}$ with $\{zw, x\overline{u}\}$ gives a smaller PDS. Otherwise, $w\in C, \overline{w}\in B$ and $\overline{w}$ has  a private neighbor $\overline{w}'$. 


In the later case, let $r$ be the neighbor of $w$ in $V\setminus S$. If $d_S(r)=3$, then replacing $\{v\overline{v}, w\overline{w}\} $ with $\{ \overline{w}\overline{w}', z\overline{v} \} $ gives a PDS with smaller $A\cup B$.

Otherwise, $d_S(r)=2$, as $r$ is not private neighbor of $w\in C$. Let $t\in S$ be the other neighbor of $r$. If $t\in A\cup B\cup C$, then $f(rw)\geq 1/2$, and $f(S_u)+ f(S_v)+ f(S_w)\geq f(x)+f(y)+f(z)+f(\overline{w}')+f(rw)\geq 4.5$. {It is safe to compute the total weight of these three pairs together, as all vertices in $S_u\cup S_v\cup S_w$, except $\overline{w}$, have degree 3. } Otherwise, $t\in D$.  If $S_t$ is a solo $D$ pair, then $f(S_u)+f(S_v)+f(S_w)+f(S_t)\geq f(x)+f(y)+f(z)+f(\overline{w}')+f(r)+1/3*3\geq 6$. 
Otherwise, $S_t$ is a linked $D$ pair.

\textbf{Subcase 1.1.1:} $\overline{t}$ has degree at least 2 in $S$.

If $r\overline{w}'\not\in E(G)$, then replacing $\{w\overline{w},t\overline{t}\}$ with $\{rt,\overline{w}\overline{w}' \}$ gives a PDS with smaller $\lambda$. If $r\overline{w}'\in E(G)$, then replacing $\{u\overline{u}, v\overline{v},w\overline{w},t\overline{t}\}$ with $\{rt,zw, x\overline{u}\}$ gives a smaller PDS, or there is a vertex $q\in V\setminus S$, such that $N_S(q)=\{\overline{t},\overline{w}\}$. In the later case, replacing $\{ v\overline{v},w\overline{w},t\overline{t}\}$ with $\{rt,z\overline{v}, \overline{w}q\}$ gives a PDS with smaller $\lambda$. 

\textbf{Subcase 1.1.2:} $t$ has degree at least 2 in $S$.

{ If both edges in $[\overline{t},V\setminus S]$ have weight at least 5/12, then $f(S_u)+f(S_v)+f(S_w)+f(S_t)\geq f(x)+f(y)+f(z)+f(\overline{w}')+f(r)+1/6+5/12*2=6$. Otherwise, at least one neighbor $q$ of $t$ has another $D$ neighbor $w_1$. Then replacing $\{t\overline{t}, w_1\overline{w_1}\}$ with $\{qw_1\}$ gives a smaller PDS, or there is a vertex $t_1\in V\setminus S$, such that $N_S(t_1)=\{\overline{t},\overline{w_1}\}$. In the later case, $f(S_u)+f(S_v)+f(S_w)+f(S_t)\geq f(x)+f(y)+f(z)+f(\overline{w}')+f(r)+1/6+1/2+1/3=6$.} 

\textbf{Subcase 1.2:} $v$ and $\overline{v}$ have different neighbors in $V\setminus  S$. Let $z$ be the neighbor of $v\in S$, and $w$ the neighbor of $\overline{v}$.
Note that $x$ and $y$ already provide weight 2 for $S_u\cup S_v$.

If $f(vz)+f(\overline{v}w) \geq 1$, then $f(S_u)+f(S_v)\geq 3$. 
Otherwise, $w$ or $z$ has a neighbor in $D\setminus S_v$. W. o. l. g. assume $w$ has a neighbor $w_1\in D\setminus S_v$. Then replacing $ \{u\overline{u}, v\overline{v},w_1\overline{w_1}\} $ with $\{x\overline{u},ww_1\}$ gives a smaller PDS.


\textbf{Case 2:} $u$ and $\overline{u}$ do not have any common neighbor in $V\setminus  S$.

For this case, we first give the following lemma, which will be used many times.

\begin{definition}
Let $S_u$ be a solo $D$ pair, with $uy,\overline{u}x\in E(G)$. Suppose $x$ has a neighbor $w_1$ in $D$, and $y$ has a neighbor $w_2$ in $D$. If $w_1\neq w_2$ and they are not paired in $G[S]$, then we call $\{w_1,w_2\}$ a \textbf{constraint pair of $S_u$}.
\end{definition}

\begin{lemma}\label{lemma7}
If $\{w_1,w_2\}$ is a constraint pair of $S_u$, then there is a vertex $t\in V\setminus  S$, such that $N_S(t)\subseteq \{u,\overline{u},\overline{w_1},\overline{w_2}\}$.
\end{lemma}
\begin{proof}
Observe that replacing $\{ u\overline{u},w_1\overline{w_1},w_2\overline{w_2} \}$ with $\{xw_1,yw_2 \}$ gives a smaller PDS, or there is a vertex $t\in V\setminus  S$, such that $N_S(t)\subseteq \{u,\overline{u},\overline{w_1},\overline{w_2}\}$.
\end{proof}

\begin{lemma}\label{lemma8}
Let $S_u$ be a solo $D$ pair, such that $u$ and $\overline{u}$ do not have common neighbor, and edges in $[S_u,V\setminus S]$ have weight not equal to 1/2.
If $\overline{w_1}$ or $\overline{w_2}$ already has degree 3, then  $\{w_1, w_2\}$ is not a constraint pair.
\end{lemma}
\begin{proof}
Suppose on the contrary, $\{w_1, w_2\}$ is a constraint pair. Then there is a vertex $t\in V\setminus S$, such that $N_S(t)\subseteq \{ u,\overline{u}, \overline{w_1}, \overline{w_2} \}$. If $\overline{w_1}$ already has degree 3, then $N_S(t)\subseteq \{u,\overline{u},\overline{w_2}\}$. But $u$ and $\overline{u}$ have no common neighbor, it is only possible that $N_S(t)\subseteq \{u,\overline{w_2}\}$ or $N_S(t)\subseteq \{\overline{u},\overline{w_2}\}$.
As edges in $[S_u,V\setminus S]$ have weight not equal 1/2, it is not possible that $N_S(t)= \{u,\overline{w_2}\}$ or $N_S(t)=\{\overline{u},\overline{w_2}\}$. Moreover, $t$ is not a private neighbor of any vertex in $\{u,\overline{u},\overline{w_2}\}$, which is a subset of $D$. Above all, it is not possible for the existence of $t$. The argument also works for the case when $\overline{w_2}$ already has degree 3. According to Lemma~\ref{lemma7}, $\{w_1, w_2\}$ can not be a constraint pair.
\end{proof}

In the remaining part of this section, we give a detailed discussion of the two possible  cases that $f(S_u)<3/2$ when $u$ and $\overline{u}$ have no common neighbor.

\subsubsection{Three 1/3 and one 5/12}
Suppose $u$ has two neighbors $y$ and $z$ such that $f(uy)=5/12$, $f(uz)=1/3$, and $\overline{u}$ has two neighbors $x$ and $w$, such that $f(\overline{u}x)=f(\overline{u}w)=1/3$.

\textbf{Case 1:} $N_S(x)=\{\overline{u},w_1,\overline{w_1}\}$, where $w_1$ and $\overline{w_1}$ are paired in $S$.

Now we discuss according to whether $y$ and $x$ have same neighbor in $D$.

\textbf{Subcase 1.1:} $N_D(y)=\{u, w_2\}$ with $w_2\not\in \{w_1,\overline{w_1}\}$.

Both $\{w_1,w_2\}$ and $\{\overline{w_1},w_2\}$ are constraint pairs of $S_u$,
thus according to Lemma~\ref{lemma7}, there are two vertices $t_1,t_2\in V\setminus  S$, such that $N_S(t_1)\subseteq \{u,\overline{u},\overline{w_1},\overline{w_2}\}$ and $N_S(t_2)\subseteq  \{u,\overline{u},w_1,\overline{w_2}\}$.

Because $u$ and $\overline{u}$ have no common neighbor, there are three possibilities of  $N_S(t_1)$: $N_S(t_1)=\{\overline{u},\overline{w_1},\overline{w_2}\}$, $N_S(t_1)=\{u,\overline{w_1},\overline{w_2}\}$ or $N_S(t_1)=\{\overline{w_1},\overline{w_2}\}$.

If $N_S(t_1)=\{u,\overline{w_1},\overline{w_2}\}$, then $N_S(t_2)= \{ \overline{u},w_1,\overline{w_2}\}$ or
$N_S(t_2)= \{w_1,\overline{w_2}\}$. If $N_S(t_2)=  \{\overline{u},w_1,\overline{w_2}\}$, then replacing $\{ u\overline{u}, w_1\overline{w_1}, w_2\overline{w_2}\}$ with $\{ t_1\overline{w_1},\overline{w_2}t_2\}$ gives a smaller PDS. If  $N_S(t_2)= \{w_1,\overline{w_2}\}$, then let $w$ be the other neighbor of $\overline{u}$, which has a neighbor $w_3\neq w_2$ in $D$. Then $\{w_2,w_3\}$ is a constraint pair of $S_u$, which is not possible according to Lemma~\ref{lemma8}, as $\overline{w_2}$ already has degree 3.

If $N_S(t_1)=\{\overline{u},\overline{w_1},\overline{w_2}\}$, then $N_S(t_2)=  \{u,w_1,\overline{w_2}\}$ or
$N_S(t_2)= \{w_1,\overline{w_2}\}$.  The case that $N_S(t_2)=  \{u,w_1,\overline{w_2}\}$ has been discussed above. If $N_S(t_2)= \{w_1,\overline{w_2}\}$, then let $w_3$ be the other $D$ neighbor of $z$.  Then $\{w_1,w_3\}$ is a constraint pair of $S_u$, which is not possible according to Lemma~\ref{lemma8}, as $\overline{w_1}$ already has degree 3.

If $N_S(t_1)=\{\overline{w_1},\overline{w_2}\}$, then there are three possibilities of  $N_S(t_2)$: $N_S(t_2)=\{\overline{u},w_1,\overline{w_2}\}$, $N_S(t_1)=\{u,w_1,\overline{w_2}\}$ or $N_S(t_2)=\{w_1,\overline{w_2}\}$. We just need to consider the case that  $N_S(t_2)=\{w_1,\overline{w_2}\}$ as the other two cases have been discussed above. Let $w_3\neq w_2$  be a $D$ neighbor of $z$. Then $\{w_1,w_3\}$ is a constraint pair of $S_u$, which is not possible according to Lemma~\ref{lemma8}, as $\overline{w_1}$ already has degree 3.

\textbf{Subcase 1.2:} $N_D(y)=\{u, w_2\}$ with $w_2\in \{w_1,\overline{w_1}\}$.

In this case, if $w$ has two $D$ neighbors that are paired with each other, then $y$ does not have any common neighbor with $w$, and the discussion of Subcase 1.1 applies.

Consider the case  $N_S(w)=\{\overline{u}, w_3,w_4\}$, where $w_3$ and $w_4$ are not paired in $S$. W. o. l. g. suppose $yw_1\in E(G)$.

If $\overline{w_1}\in \{w_3,w_4\}$, say $\overline{w_1}=w_3$, then $\{w_1,w_4\}$ is a constraint pair of $S_u$, which is not possible according to Lemma~\ref{lemma8}, as $\overline{w_1}$ already has degree 3. Otherwise, both $\{w_1, w_3\}$ and $\{w_1, w_4\}$ are constraint pairs of $S_u$, which is not possible according to Lemma~\ref{lemma8}, as $\overline{w_1}$ already has degree 2.

\textbf{Case 2:} The other two neighbors of $x$, $w$ are not paired in $S$.

Suppose $N_D(x)=\{\overline{u},w_1, w_2\}$, $N_D(y)=\{u,w_3\}$.
 Then either $w_3\not\in S_{w_1}$ or $w_3\not\in S_{w_2}$.

W.o.l.g., assume $w_3\not\in S_{w_1}$, then according to Lemma~\ref{lemma7}, there is a vertex $t_1\in V\setminus  S$, such that $N_S(t_1)\subseteq\{u,\overline{u},\overline{w_1},\overline{w_3}\}$.

Because $u$ and $\overline{u}$ do not have common  neighbor, there are three possibilities of  $N_S(t_1)$: $N_S(t_1)=\{u,\overline{w_1},\overline{w_3}\}$, $N_S(t_1)=\{\overline{u},\overline{w_1},\overline{w_3}\}$, or $N_S(t_1)=\{\overline{w_1},\overline{w_3}\}$.

If $N_S(t_1)=\{u,\overline{w_1},\overline{w_3}\}$, then according to Lemma~\ref{lemma7}, there is a vertex $t_2\in V\setminus  S$, such that $N_S(t_2)\subseteq\{\overline{u},\overline{w_1},w_3\}$. If $N_S(t_2)=\{\overline{u},\overline{w_1},w_3\}$, then replacing  $\{ u\overline{u}, w_1\overline{w_1},w_3\overline{w_3}  \}$   with $\{  t_1\overline{w_3},\overline{w_1}t_2\}$ gives a smaller PDS.
If $N_S(t_2)=\{\overline{w_1},w_3\}$, then $w$ or $x$ has a $D$ neighbor $w_4\not\in S_{w_3}\cup S_{w_1}$, such that $\{w_4,\overline{w_3}\}$ is a constraint pair of $S_u$, which is not possible according to Lemma~\ref{lemma8}, as $w_3$ already has degree 3.

If $N_S(t_1)=\{\overline{u},\overline{w_1},\overline{w_3}\}$, then according to Lemma~\ref{lemma7}, there is a vertex $t_2\in V\setminus  S$, such that $N_S(t_2)\subseteq\{u,w_1,\overline{w_3}\}$. If $N_S(t_2)=\{u,w_1,\overline{w_3}\}$, then $\{\overline{w_1},\overline{w_3}\}$ is a constraint pair, which is not possible according to Lemma~\ref{lemma8}, as $w_1$ already has degree 3. If $N_S(t_2)=\{w_1,\overline{w_3}\}$, then $z$ has a $D$ neighbor $w_4\not\in S_{w_1}$. In this case, $\{\overline{w_1},w_4\}$ is a constraint pair of $S_u$, which is not possible, as $w_1$ already has degree 3.

If $N_S(t_1)=\{\overline{w_1},\overline{w_3}\}$, then look at the other $D$ neighbor of $x$.

If $w_2=w_3$, $z\overline{w_3}\in E(G)$(or $w\overline{w_3}\in E(G)$), then $z$(or $w$) has a $D$ neighbor $w_4\neq w_1$, such that $\{w_3,w_4\}$ is a constraint pair, which is not possible according to Lemma~\ref{lemma8}, as $\overline{w_3}$ already has degree 3. Otherwise, $z\overline{w_3}\not\in E(G)$ and $w\overline{w_3}\not \in E(G)$, then $z$ and $w$ together have at least two neighbors $w_4$ and $w_5$ in $D\setminus \{w_1\}$, such that both $\{w_3,w_4\}$ and $\{w_3,w_5\}$ are constraint pairs, which is not possible according to Lemma~\ref{lemma8}, as $\overline{w_3}$ already has degree 2.

If $w_2=\overline{w_3}$, then look at the neighbor of $w$. If $w$ has a $D$ neighbor $w_4\not\in \{w_1, w_3\}$, then $\{w_3,w_4\}$ is a constraint pair, which is not possible according to Lemma~\ref{lemma8}, as $\overline{w_3}$ already has degree 3. Otherwise, $ww_1\in E(G)$ and $ww_3\in E(G)$. Then $z$ has a $D$ neighbor $w_5$, which forms a constraint pair with $w_3$, which is not possible according to Lemma~\ref{lemma8}, as $\overline{w_3}$ already has degree 3.

Otherwise, $\{w_2,w_3\}$ is a constraint pair, and there is a vertex $t_2\in V\setminus S$, such that $N_S(t_2)\subseteq \{u,\overline{u},\overline{w_2},\overline{w_3}\}$. There are three possibilities of  $N_S(t_2)$: $N_S(t_2)= \{u,\overline{w_2},\overline{w_3}\}$, $N_S(t_2)= \{\overline{u},\overline{w_2},\overline{w_3}\}$ or $N_S(t_2)= \{\overline{w_2},\overline{w_3}\}$.

If $N_S(t_2)= \{u,\overline{w_2},\overline{w_3}\}$, then both $\{ w_1, \overline{w_2}\}$ and $\{w_1, \overline{w_3} \}$ are constraint pairs of $S_u$, which is not possible according to Lemma~\ref{lemma8}, as $\overline{w_1}$ already has degree 2.

If $N_S(t_2)= \{\overline{u},\overline{w_2},\overline{w_3}\}$, then $\{\overline{w_2},w_3\}$ is a constraint pair of $S_u$, which is not possible according to Lemma~\ref{lemma8}, as $\overline{w_3}$ already has degree 3.

If $N_S(t_2)= \{\overline{w_2},\overline{w_3}\}$, then look at the $D$ neighbor of $w$. If $w$ has a $D$ neighbor $w_4\not\in \{w_1, w_2,w_3\}$, then $\{w_3,w_4\}$ is a constraint pair of $S_u$, which is not possible according to Lemma~\ref{lemma8}, as $\overline{w_3}$ already has degree 3. Otherwise, $N_D(w)\subseteq \{w_1,w_2,w_3\}\cup \{\overline{u}\}$. In this case, if $ww_3\in E(G)$, then $z$ has a $D$ neighbor $w_5 \not\in \{w_1,w_2\}$, which forms a constraint pair with $w_3$, which is not possible according to Lemma~\ref{lemma8}, as $\overline{w_3}$ already has degree 3. Otherwise, $N_D(w)=\{\overline{u},w_1,w_2\}$. In this case, replacing $\{u\overline{u},w_1\overline{w_1},w_3\overline{w_3}\}$ with $\{uy,t_1\overline{w_1}\}$ gives a smaller PDS.

\subsubsection{Four 1/3}

\textbf{Case 1:}$N_D(x)=\{\overline{u},w_1,\overline{w_1}\}$, where $w_1$ and $\overline{w_1}$  are paired in $G[S]$. 


\textbf{Subcase 1.1:} $y$ has a $D$ neighbor $w_2\not\in S_{w_1}$.

In this case, $\{\overline{w_1},\overline{w_2}\}$ is a constraint pair of $S_u$, thus according to Lemma~\ref{lemma7}, there is a vertex $t_1\in V\setminus S$, such that $N_S(t_1)\subseteq \{u,\overline{u},\overline{w_1},\overline{w_2}\}$. As $u$ and $\overline{u}$ do not have common neighbor, and all edges between $S_u$ and $V\setminus S$ have weight 1/3, there are three possibilities of  $N_S(t_1)$: $N_S(t_1)= \{u,\overline{w_1},\overline{w_2}\}$, $N_S(t_1)= \{\overline{u},\overline{w_1},\overline{w_2}\}$, or $N_S(t_1)= \{\overline{w_1},\overline{w_2}\}$.

If $N_S(t_1)= \{u,\overline{w_1},\overline{w_2}\}$, then $\{w_1,\overline{w_2}\}$ is also a constraint pair of $S_u$, which is not possible, according to Lemma~\ref{lemma8}, as $\overline{w_1}$ already has degree 3.


If $N_S(t_1)= \{\overline{u},\overline{w_1},\overline{w_2}\}$, then $\{\overline{w_1},w_2\}$ is also a constraint pair of $S_u$. Thus according to Lemma~\ref{lemma7}, there is a vertex $t_2\in V\setminus S$, such that $N_S(t_2)\subseteq \{ w_1, u,\overline{w_2}\}$. There are two possibilities of  $N_S(t_2)$: $N_S(t_2)= \{ w_1, u,\overline{w_2}\}$, or $N_S(t_2)= \{ w_1, \overline{w_2}\}$. If $N_S(t_2)= \{ w_1, u,\overline{w_2}\}$, then $\{w_1,\overline{w_2}\}$ is a constraint pair, which is not possible, as $\overline{w_1}$ already has degree 3. 
If $N_S(t_2)= \{ w_1, \overline{w_2}\}$, then as $w_1$ and $\overline{w_1}$ already has degree 3, $z$ has a $D$ neighbor $w_3\not\in S_{w_1}\cup S_{w_2}$. Thus $\{w_1, w_3\}$ is a constraint pair of $S_u$, 
which is not possible according to Lemma~\ref{lemma8}, as $\overline{w_1}$ already has degree 3.

If $N_S(t_1)= \{\overline{w_1},\overline{w_2}\}$, then $\{\overline{w_1},w_2\}$ is also a constraint pair of $S_u$. Thus according to Lemma~\ref{lemma7}, there is a vertex $t_2\in V\setminus S$, such that $N_S(t_2)\subseteq \{ w_1, u,\overline{u},\overline{w_2}\}$. There are three possibilities of  $N_S(t_2)$: $N_S(t_2)= \{u,w_1,\overline{w_2}\}$, $N_S(t_2)= \{\overline{u},w_1,\overline{w_2}\}$, or $N_S(t_2)= \{w_1,\overline{w_2}\}$.
The cases of $N_S(t_2)= \{u,w_1,\overline{w_2}\}$ and $N_S(t_2)= \{\overline{u},w_1,\overline{w_2}\}$ have been discussed above. If $N_S(t_2)= \{w_1,\overline{w_2}\}$, then as $w_1$ and $\overline{w_1}$ already has degree 3, $z$ has a $D$ neighbor $w_3\not\in S_{w_1}\cup S_{w_2}$.
Thus $\{w_1, w_3\}$ is a constraint pair of $S_u$, which is not possible according to Lemma~\ref{lemma8}, as $\overline{w_1}$ already has degree 3.

\textbf{Subcase 1.2:} If $N_D(y)=\{u\}\cup S_{w_1}$.

In this case, $w$ has a neighbor $w_2$ not in $S_{w_1}$, as both vertices in $S_{w_1}$ already has degree 3.
Thus, $\{w_1,w_2\}$ is a constraint pair of $S_u$, which is not possible according to Lemma~\ref{lemma8}, as $\overline{w_1}$ already has degree 3.

\textbf{Case 2:} none of the two neighbors in $D\setminus S_u$ of $x,y,z$ and $w$ are paired in $S$.

Let $N_S(x)\setminus\{\overline{u}\}=\{w_1,w_2\}$, $N_S(y)\setminus\{u\}=\{w_3,w_4\}$.

As $w_3$ and $w_4$ are not paired in $G[S]$, at least one of them is not in $S_{w_1}$. W. o. l. g., assume $w_3\not\in S_{w_1}$, then $\{w_1,w_3\}$ is a constraint pair of $S_u$. Thus there is a vertex $t_1\in V\setminus S$, such that $N_S(t_1)\subseteq \{u,\overline{u}, \overline{w_1}, \overline{w_3}\}$. As $u$ and $\overline{u}$ have no common neighbor, there are three possibilities of  $N_S(t_1)$: $N_S(t_1)=\{u, \overline{w_1}, \overline{w_3}\}$, $N_S(t_1)=\{\overline{u}, \overline{w_1}, \overline{w_3}\}$, or $N_S(t_1)=\{\overline{w_1}, \overline{w_3}\}$.

If $N_S(t_1)=\{u,\overline{w_1}, \overline{w_3}\}$, then $\{w_1,\overline{w_3}\}$ is also a constraint pair of $S_u$. Thus, there is a vertex $t_2\in V\setminus S$, such that $N_S(t_2)\subseteq \{\overline{u}, \overline{w_1}, w_3\}$. There are two possibilities of  $N_S(t_2)$: $N_S(t_2)= \{\overline{u}, \overline{w_1}, w_3\}$ or $N_S(t_2)= \{\overline{w_1}, w_3\}$. 
If $N_S(t_2)= \{\overline{u}, \overline{w_1}, w_3\}$, then look at $y$ and $w_1$. If $yw_1\not\in E(G)$, then  $w_4\not\in S_{w_1}\cup S_{w_3}$. Thus $\{w_1,w_4\}$ is a constraint pair, which is not possible according to Lemma~\ref{lemma8}, as $\overline{w_1}$ already has degree 3.
Thus $yw_1\in E(G)$. In this case, if $x\overline{w_3}\not\in E(G)$, then $w_2\not\in S_{w_3}$. And so $\{w_2,\overline{w_3}\}$ is a constraint pair, which is not possible as $w_3$ already has degree 3.  Therefore, $x\overline{w_3},yw_1\in E(G)$, then the graph is the \textbf{Petersen Graph}. If $N_S(t_2)= \{\overline{w_1}, w_3\}$, then either $w_2\neq \overline{w_3}$ or $w$ has a $D$ neighbor $w_5\neq \overline{w_3}$. Therefore $\overline{w_3}$ forms a constraint pair with $w_2$ or $w_5$, which is not possible, as $w_3$ already has degree 3.

If $N_S(t_1)=\{\overline{u}, \overline{w_1}, \overline{w_3}\}$, then $\{\overline{w_1},w_3\}$ is also a constraint pair. Thus there is a vertex $t_2\in V\setminus S$, such that $N_S(t_2)\subseteq \{u, w_1, \overline{w_3}\}$. There are two possibilities of  $N_S(t_2)$: $N_S(t_2)= \{u, w_1, \overline{w_3}\}$ or $N_S(t_2)= \{w_1, \overline{w_3}\}$. If $N_S(t_2)= \{u, w_1, \overline{w_3}\}$, then $\{\overline{w_1},\overline{w_3} \}$ is a constraint pair, which is not possible according to Lemma~\ref{lemma8}, as $w_1$ already has degree 3. If $N_S(t_2)= \{w_1, \overline{w_3}\}$, and $w_4=\overline{w_1}$, then $\{\overline{w_1},w_2\}$ is a constraint pair, which is not according to Lemma~\ref{lemma8}, as $w_1$ already has degree 3. Otherwise, $N_S(t_2)= \{w_1, \overline{w_3}\}$, and $w_4\not\in S_{w_1}\cup S_{w_3}$. Then $\{\overline{w_1},w_4\}$ is a constraint pair, which is not possible according to Lemma~\ref{lemma8}, as $w_1$ already has degree 3.

If $N_S(t_1)=\{\overline{w_1}, \overline{w_3}\}$, and $xw_3\in E(G)$, i.e. $w_2=w_3$, then there is a $D$ neighbor $w_5$ of $z$, such that $w_5\not\in S_{w_3}$, as the two neighbors of $z$ are not paired in $G[S]$. In this case, $\{w_3,w_5\}$ is a constraint pair, thus there is a vertex $t_2\in V\setminus S$, such that $N_S(t_2)\subseteq \{ \overline{u}, \overline{w_3}, \overline{w_5} \}$. There are two possibilities of  $N_S(t_2)$: $N_S(t_2)= \{ \overline{u}, \overline{w_3}, \overline{w_5} \}$ or $N_S(t_2)= \{ \overline{w_3}, \overline{w_5} \}$. If $N_S(t_2)= \{ \overline{u}, \overline{w_3}, \overline{w_5} \}$, then $\{\overline{w_3}, w_5\}$ is a constraint pair, which is not possible according to Lemma~\ref{lemma8}, as $w_3$ already has degree 3. If $N_S(t_2)= \{ \overline{w_3}, \overline{w_5} \}$, then $z$ has a neighbor $w_6\not\in S_{w_3}\cup S_{w_5}$, and $\{w_3, w_6\}$ is a constraint pair of $S_u$, which is not possible according to Lemma~\ref{lemma8}, as $\overline{w_3}$ already has degree 3.

Otherwise, $N_S(t_1)=\{\overline{w_1}, \overline{w_3}\}$ and $xw_3\not\in E(G)$. If $x\overline{w_3}\in E(G)$, i.e. $w_2=\overline{w_3}$,  then  $w$ has a neighbor $w_7$ not in $S_{w_3}$. In this case,  $\{w_3,w_7\}$ is a constraint pair, which is not possible according to Lemma~\ref{lemma8}, as $\overline{w_3}$ already has degree 3. Otherwise, $w_2\not\in S_{w_3}$. Thus $\{w_2, w_3\}$ is a constraint pair of $S_u$. According to Lemma~\ref{lemma7}, there is a vertex $t_2\in V\setminus S$, such that $N_S(t_2)\subseteq \{u,\overline{u}, \overline{w_2},\overline{w_3}\}$. There are three possibilities of  $N_S(t_2)$: $N_S(t_2)= \{u, \overline{w_2},\overline{w_3}\}$, $N_S(t_2)= \{\overline{u}, \overline{w_2},\overline{w_3}\}$, or $N_S(t_2)= \{ \overline{w_2},\overline{w_3}\}$.
If $N_S(t_2)= \{u, \overline{w_2},\overline{w_3}\}$, then both $\{w_1,\overline{w_3}\}$ and $\{w_1,\overline{w_2}\}$ are constraint pairs of $S_u$, which is not possible, as $\overline{w_1}$ already has degree 2.
If $N_S(t_2)= \{\overline{u}, \overline{w_2},\overline{w_3}\}$, then $\{\overline{w_2},w_3\}$ is a constraint pair, which is not possible according to Lemma~\ref{lemma8}, as $\overline{w_3}$ already has degree 3.
If $N_S(t_2)= \{ \overline{w_2},\overline{w_3}\}$, then $ww_1\in E(G)$ and $ww_2\in E(G)$. Otherwise, $w$ has a neighbor $w_5\in D$, such that $\{w_5,w_3\}$ is a constraint pair of $S_u$, which is not possible according to Lemma~\ref{lemma8}, as $\overline{w_3}$ already has degree 3. Thus consider the case when $N_S(t_1)=\{\overline{w_1}, \overline{w_3}\},N_S(t_2)= \{ \overline{w_2},\overline{w_3}\}$ and $ww_1\in E(G)$ and $ww_2\in E(G)$. Look at the last $D$ neighbor $z$ of $u$. Symmetrically, $y$ and $z$ also have same neighborhood $\{u, w_3, w_4\}$. Thus there are two vertices $t_3,t_4 \in V\setminus S$, such that
$N_S(t_3)=\{\overline{w_1},\overline{w_4}\}, N_S(t_4)=\{\overline{w_2},\overline{w_4}\}$. In this case, replacing $\{w_1\overline{w_1},w_2\overline{w_2},w_3\overline{w_3},w_4\overline{w_4}\}$ with $\{ t_3\overline{w_1}, t_2\overline{w_2},yw_3 \} $ gives a smaller PDS, a contradiction.

Above all, we have proved that if $G$ is not the Petersen Graph, then each component in $G[S]$ has enough total weight, such that averagely each pair in $S$ has weight at least 3/2. Thus Theorem~\ref{thm7} is correct.

\section{Conclusion}
In this paper, we study the paired domination problem in cubic graphs. We show that, except the Petersen Graph, $\gamma_{pr}(G)\leq 4n/7$ holds for all cubic graphs. Our result confirms the remaining case of Conjecture 2 for regular graphs.
As our result does not require the graph to be claw-free, it is a promising step towards proving Conjecture 2 completely.

\section{Acknowledgement}
Bin Sheng was supported by National Natural Science Foundation of China (No. 61802178).  Changhong Lu was supported by National Natural Science Foundation of China (No. 11871222) and Science and Technology Commission of
Shanghai Municipality (No. 18dz2271000, 19jc1420100).

\bibliographystyle{alpha}
\bibliography{references}
\end{document}